\documentclass[11pt]{amsart}

\usepackage{amsfonts}
\usepackage{amssymb}
\usepackage{graphicx}
\usepackage{amsmath,mathtools}
\usepackage{latexsym}
\usepackage{xypic}
\usepackage{mathrsfs}
\usepackage{enumitem} 
\usepackage{braket}
\usepackage[margin=1.5in]{geometry}
\usepackage{esint}
\usepackage{dsfont}
\usepackage{amsthm}
\usepackage{color}

\usepackage{hyperref}

\usepackage{cleveref}

\setlist{itemsep=3pt}

\allowdisplaybreaks

\newtheorem{proposition}{Proposition}
\newtheorem{theorem}[proposition]{Theorem}

\newtheorem{corollary}[proposition]{Corollary}
\newtheorem{conjecture}[proposition]{Conjecture}

\theoremstyle{definition}
\newtheorem{definition}[proposition]{Definition}

\newtheorem{remark}[proposition]{Remark}

\newcommand{\HH}{\mathbb{H}}

\newcommand{\RR}{\mathbb{R}}

\newcommand{\ZZ}{\mathbb{Z}}

\newcommand{\bH}{\mathbf{H}}

\newcommand{\bR}{\mathbf{R}}
\newcommand{\bS}{\mathbf{S}}

\newcommand{\cC}{\mathcal C}

\newcommand{\cG}{\mathcal G}

\DeclareMathOperator{\tr}{tr}

\DeclareMathOperator{\area}{area}
\DeclareMathOperator{\Ric}{Ric}

\DeclareMathOperator{\diam}{diam}

\DeclareMathOperator{\vol}{vol}

\DeclareMathOperator{\sys}{sys}
\DeclareMathOperator{\biRic}{biRic}

\usepackage{graphicx}

\title{Minimal surfaces and comparison geometry}
\author{Otis Chodosh}
\address{Department of Mathematics, Stanford University, Building 380, Stanford, CA 94305, USA}
\email{ochodosh@stanford.edu}

\begin{document}

\maketitle

\section{Introduction.}
A basic idea in Riemannian geometry is to minimize the length of a curve to get a stable geodesic. By analyzing the second variation of length, it is possible to relate curvature to global properties. A natural generalization is to minimize the area of a submanifold.  We discuss comparison results that have been obtained via this method, as well as some recent progress on regularity issues. 

\section{Minimization of area.}

Foundational work in geometric measure theory by Almgren, De Giorgi, Federer, Fleming, and Simons in the 1960s gives the following existence and regularity for codimension-one minimizers:
\begin{theorem}[{cf.\ \cite{Simon:GMT}}]\label{theo:minimize}
Let $(M^{n+1},g)$ be a closed Riemannian manifold and $\sigma \in H_n(M;\ZZ)$ be a codimension-one homology class. If $n+1\leq 7$ then there's a smooth representative $\Sigma \in \sigma$ of least area.
\end{theorem}
In general, $\Sigma$ could have a singular set of dimension $\leq n-7$.  

\section{First and second variation.} 
Consider now $\Sigma^n \subset (M^{n+1},g)$ of least area. As with curves, we consider the first and second derivatives of the area of a variation of $\Sigma$ through hypersurfaces $t\mapsto \Sigma_t$. Assume that $\Sigma$ is \emph{two-sided} (i.e.\ admits a unit-normal $\nu$) and the variation has initial speed $\varphi$. The \emph{first variation of area} gives
\begin{equation}\label{eq:first-var}
0 = \frac{d}{dt}\Big|_{t=0} \area_g(\Sigma_t) =  \int_{\Sigma} H \varphi 
\end{equation}
where $H$ is the scalar mean curvature of $\Sigma \subset (M,g)$. We recall that $H=\tr A$ is the trace of the second fundamental form $A(X,Y) = g(\nabla_X \nu,Y)$. Since \eqref{eq:first-var} holds for any $\varphi$, we find that $H=0$, i.e.\ $\Sigma$ is a \emph{minimal surface}. As with geodesics, we should also consider the second derivative. The \emph{second variation of area} yields
\[
0 \leq \frac{d^2}{dt^2}\Big|_{t=0} \area_g(\Sigma_t) =  \int_{\Sigma} |\nabla \varphi|^2 - (|A|^2 + \Ric(\nu,\nu)) \varphi^2. 
\]
Here, $|A|$ is the norm of the second fundamental form and $\Ric(\nu,\nu)$ is the ambient normal Ricci curvature.

\section{The Geroch conjecture}  
We call a two-sided $\Sigma^n \subset (M^{n+1},g)$ a \emph{stable minimal hypersurface} if
\begin{equation}\label{eq:stable-min}
  \int_{\Sigma}  (|A|^2 + \Ric(\nu,\nu)) \varphi^2 \leq \int_\Sigma |\nabla \varphi|^2 
\end{equation}
for any $\varphi \in C^\infty_c(\Sigma)$. For applications, we must know that certain curvature properties of $(M,g)$ will restrict the geometry/topology of $\Sigma$. The first such result was obtained in 1979 by Schoen and Yau who showed:
\begin{theorem}[{\cite{SY:3d-torus}}]\label{theo:psc-min-3d-sphere}
If $(M^3,g)$ has positive scalar curvature $R>0$ then each component of a closed two-sided stable minimal surface $\Sigma\subset (M,g)$ has genus zero. 
\end{theorem}
\begin{proof}
The Gauss equations relate the curvature of $g$ to the curvature of the induced metric on $\Sigma$. In particular, using $H=0$ and the (doubly traced) Gauss equations we obtain 
\begin{equation}\label{eq:traced-gauss-min-surf}
\Ric(\nu,\nu) + |A|^2 = \frac 12 (R + |A|^2 - 2K)
\end{equation}
where $R$ is the ambient scalar curvature and $K$ is the intrinsic Gaussian curvature. Using this in \eqref{eq:stable-min} we have
\begin{equation}\label{eq:2nd-var-rearrange-surf}
 \frac 12 \int_\Sigma(R+|A|^2) \varphi^2 \leq \int_\Sigma |\nabla \varphi|^2 + K \varphi^2.
\end{equation}
Taking $\varphi=1$ on $\Sigma$ in \eqref{eq:2nd-var-rearrange-surf}, the Gauss--Bonnet theorem and $R+|A|^2>0$ yields $\chi(\Sigma) > 0$, proving the assertion.
\end{proof}

As a consequence, Schoen and Yau were able to resolve the \emph{Geroch conjecture}:
\begin{theorem}[{\cite{SY:3d-torus}}]\label{thm:geroch}
The $3$-torus does not admit a metric of positive scalar curvature. 
\end{theorem}
\begin{proof}
Assume that $(T^3,g)$ has positive scalar curvature. Let $\sigma = [T^2 \times \{0\}] \in H_2(T^3;\ZZ)$. Theorem \ref{theo:minimize} yields $\Sigma \in \sigma$ smooth and area-minimizing with respect to the metric $g$. The area-minimizing property and topological considerations imply that $\Sigma$ is a two-sided stable minimal surface, so Theorem \ref{theo:psc-min-3d-sphere} implies that each component of $\Sigma$ is a sphere. On the other hand, $\pi_2(T^3) = 0$ so $[\Sigma]=0\in H_2(T^3;\ZZ)$. This is a contradiction. 
\end{proof}

An alternative proof of Theorem \ref{thm:geroch} was given by Gromov and Lawson \cite{GL:complete} based on the Dirac operator.   Recently, Stern has found another approach \cite{Stern} to Theorem \ref{thm:geroch} via analysis of the level sets of a harmonic function. The methods used in these proofs have had several important consequences including the positive mass theorem in general relativity \cite{SY:PMT1,Witten:positivemass} and Schoen's resolution of the Yamabe problem \cite{Schoen:yamabe}.

\section{Inductive descent.}\label{sec:induct}
Theorem \ref{thm:geroch} actually holds in all dimensions as proven by Schoen and Yau (using minimal surfaces) and Gromov and Lawson (using the Dirac operator). 
\begin{theorem}[{\cite{SY:manuscripta,GromovLawson}}]\label{thm:geroch-high-dim}
The $(n+1)$-torus does not admit a metric of positive scalar curvature. 
\end{theorem}
\begin{proof}[Sketch of the proof]
Suppose that $(T^{n+1},g)$ has positive scalar curvature. Let $\Sigma_n$ minimize area in the homology class $[T^n\times \{0\}]\in H_n(T^{n+1};\ZZ)$. If $(n+1)\leq 7$ then $\Sigma_n$ is smooth by Theorem \ref{theo:minimize}. By repeating the derivation of \eqref{eq:2nd-var-rearrange-surf} without taking $\varphi=1$, one may obtain the following version of the stability inequality
\begin{equation}\label{eq:stab-inductive}
\int_{\Sigma_n} (R + |A|^2) \varphi^2 \leq \int_{\Sigma_n} 2|\nabla \varphi|^2 + R_{\Sigma_n} \varphi^2
\end{equation}
where $R$ is the ambient scalar curvature and $R_{\Sigma_n}$ is the scalar curvature of the induced metric on $\Sigma_n$. We \emph{pretend} that \eqref{eq:stab-inductive} and $R>0$ implies that $R_{\Sigma_n} > 0$ (this is explained Section \ref{sec:spect}).

Because $(dx^1 \wedge\dots \wedge dx^{n-1})|_{\Sigma} \neq 0 \in H^{n-1}(\Sigma;\RR)$ we may find $\Sigma_{n-1}\subset \Sigma_n$ minimizing area in the homology class dual to $(dx^1 \wedge\dots \wedge dx^{n-2})|_{\Sigma_n}$. Repeating this inductively we find $\Sigma_2\subset \Sigma_3\subset\dots \subset \Sigma_n$. We are pretending that each $\Sigma_k$ has $R_{\Sigma_k}>0$, so $\Sigma_2$ has genus zero by Theorem \ref{theo:psc-min-3d-sphere}. This contradicts $dx^1|_{\Sigma_2} \neq 0 \in H^1(\Sigma_2;\RR) = 0$. 
\end{proof}

\begin{remark}
This proof works for $n+1\leq 7$. Beyond this dimension, area-minimizing hypersurfaces may have extremely complicated singularities. Several methods to extend the inductive descent argument in the presence of potential singularities have been proposed. We describe one approach in Section \ref{sec:generic}. The Dirac operator method does not have dimensional restrictions but is usually only applicable to spin manifolds (including the torus) 
\end{remark}

The inductive descent method has been recently refined by  Gromov and Hanke \cite{GromovHanke} who consider torsion classes in integral homology and by Li and Zhang \cite{LiZhang:topMin} who have generalized the technique to nonorientable manifolds. 

\section{Spectral positivity of curvature.} \label{sec:spect}

In the proof of Theorem \ref{thm:geroch-high-dim} sketched above, we glazed over the inductive analysis of the stability inequality \eqref{eq:stab-inductive}. We explain this briefly here. Observe that \eqref{eq:stab-inductive} implies that (writing $\Sigma=\Sigma_n$) the operator $-2\Delta_\Sigma + R_\Sigma$ is positive. We call this condition \emph{spectral positivity of scalar curvature}. Since $-\Delta_\Sigma$ is nonnegative, this is a weaker condition than positivity of scalar curvature. 

A key principle is that spectral positivity of curvature has---in many cases---the same geometric and topological implications as pointwise positivity of curvature. For example, consider a closed surface $(\Sigma^2,h)$ with strictly positive Gaussian curvature in the spectral sense, $ -\alpha \Delta_\Sigma + K_\Sigma \geq 1$. This is equivalent to
\begin{equation}\label{eq:spect-pos-Gauss}
  \int_\Sigma \varphi^2 \leq \int_\Sigma \alpha |\nabla \varphi|^2 + K_\Sigma \varphi^2 
\end{equation}
for all $\varphi \in C^\infty(\Sigma)$. Then, this principle says that $(\Sigma,h)$ will behave as if it has Gaussian curvature $\geq 1$. For example, if $\Sigma$ is connected we can take $\varphi = 1$ in \eqref{eq:spect-pos-Gauss} and use Gauss--Bonnet to find $\area(\Sigma) \leq 2\pi \chi(\Sigma)$. Thus $\Sigma$ is a sphere and has area $\leq 4\pi$ (the same thing as implied by $K_\Sigma \geq 1$).   In the sequel, we will often pretend (as in the proof of Theorem \ref{thm:geroch-high-dim}) that spectral positivity of curvature is the same thing as pointwise positivity of curvature. For details on how to make such arguments correct and rigorous one may refer to e.g.\ \cite{gromov2019lectures,Chodosh:minscal}.

While we didn't specify the coefficient $\alpha$, this can be extremely important. For example, Schoen and Yau proved \cite{SY:condensation} that if a complete $(\Sigma^2,h)$ satisfies \eqref{eq:spect-pos-Gauss} with $\alpha \in (0,4)$, then $\Sigma$ is compact and $ \diam(\Sigma) \leq \frac{2\pi}{\sqrt{4-\alpha}}$  (cf.\ \cite{HuXuZhang:diameter}). No such estimate is possible for $\alpha \geq 4$ since the hyperbolic plane has $K_{\HH^2} = -1$ and $-\Delta_{\HH^2} \geq -\frac 14$.

\section{Localization via prescribed mean curvature.} The minimal surface (and Dirac) techniques have the downside that area-minimization (or the construction of a harmonic spinor) is a global operation. Gromov introduced a ``$\mu$-bubble'' localization of the minimal surface technique and used it to prove:
\begin{theorem}[{\cite{Gromov:metric-inequalities,gromov1996positive,gromov2019lectures,zhu2020width}}]\label{thm:bandwidth}
For $n\leq 6$, if $(T^{n}\times [-1,1],g)$ has scalar curvature $R \geq n(n+1)$, then
\begin{equation}\label{eq:bandwidth}
d(T^n\times \{-1\},T^n\times \{1\}) < \frac{2\pi}{n+1}.
\end{equation}
\end{theorem}
Note that no assumptions are made at the boundary and the theorem relates scalar curvature and distance. 
\begin{proof}[Sketch of the proof for $n=2$]
Fix a function $h$ on the ``band'' $W : = T^2 \times [-1,1]$ and write $\partial_\pm W = T^2 \times \{\pm 1\}$. Then, for $\Omega\subset W$ containing a neighborhood of $\partial_-W$ we consider the functional
\[
\mu(\Omega) : = \area(\partial\Omega) - \int_\Omega h
\]
If a critical point of $\mu(\cdot)$ exists with $\partial\Omega = \partial_-W \cup \Sigma$, $\Sigma$ contained in the interior of $W$, the first variation of $\mu(\cdot)$ implies that the mean curvature of $\Sigma$ satisfies the prescribed mean curvature equation $H=h$. 

The second variation of $\mu(\cdot)$ becomes 
\[
 \int_\Sigma \left( 3 + \frac{3}{4} h^2 + \nabla_\nu h\right) \varphi^2 \leq \int_\Sigma |\nabla \varphi|^2 +  K \varphi^2
\]
for $K$ the Gaussian curvature of the induced metric on $\Sigma$. Thus, as long as the inequality 
\begin{equation}\label{eq:mu-bubb-potential1}
3 + \frac{3}{4} h^2 + \nabla_\nu h >0
\end{equation} 
holds, we may take $\varphi=1$ as in the proof of Theorem \ref{theo:psc-min-3d-sphere} to find that each component of $\Sigma$ is a sphere. This is a contradiction, since the projection $W\to T^2$ restricts to a degree $1$ map $\Sigma\to T^2$. 

Take $h=h(\rho)$ where $\rho=d(\cdot,\partial_-W)$ (or an appropriate smoothing). Then, since $|\nabla \rho| = 1$ a sufficient condition for \eqref{eq:mu-bubb-potential1} to hold is
\[
3 + \frac 34 h^2 - |h'| = \delta > 0. 
\]
Set $\delta=0$ for simplicity (one should actually take $\delta>0$ sufficiently small) and observe the following function is a solution to the ODE:
\[
h(\rho) = -2 \tan \left(\tfrac{3\rho}{2} - \tfrac \pi 2\right).
\]
Note that $h(\rho) \to \infty$ as $\rho\to 0$ and $-\infty$ as $\rho\to \frac{2\pi}{3}$. 

For contradiction we now assume that $d(\partial_-W,\partial_+W)> \frac{2\pi}{3}$. Now, the ``effective mean curvature'' $H-h$ of the boundary of the region $W' : = \rho^{-1}[0,\frac{2\pi}{3}] \subset W$ has the correct sign for us to construct a minimizer. This gives a contradiction using the second variation argument above. 
\end{proof}

Note that Zeidler and Cecchini have discovered a similar localization of the Dirac method \cite{Zeidler:band,Cecchini:longneck,CZ}. 

The bandwidth inequality is proven by contradiction. However, in an important (albeit trivial) observation made with C.\ Li in \cite{ChodoshLi2020generalized} is that the $\mu$-bubble technique could be used to give direct information on sufficiently long ``annular regions'' with strictly positive scalar curvature but unknown topological type. 

\begin{theorem}\label{theo:mu-bubble-annulus}
Suppose that $(W,g)$ is a compact Riemannian manifold with scalar curvature $R \geq R_0 > 0$ and $\partial_\pm W$ is a partition of components of $\partial W$ into two non-empty subsets so that $d(\partial_-W,\partial_+W) \geq L=L(n,R_0)$. Then there's a hypersurface $\Sigma\subset W$ separating $\partial_+W$ from $\partial_-W$ so that the operator $-2\Delta_\Sigma + R_\Sigma$ is positive. 
\end{theorem}

\section{The $K(\pi,1)$ conjecture.} We now discuss some recent applications of the minimal surface and $\mu$-bubble techniques to the study of scalar curvature on aspherical manifolds. An \emph{aspherical manifold} (also called a $K(\pi,1)$ manifold) has contractible universal cover. The following conjecture, made by Schoen and Yau as well as Gromov and Lawson in the 1980s would generalize Theorems \ref{thm:geroch} and \ref{thm:geroch-high-dim}:
\begin{conjecture}\label{conj:asphericalPSC}
There is no Riemannian metric of positive scalar curvature on a closed aspherical manifold. 
\end{conjecture} 
As explained by Gromov in \cite[\S 4, 16]{gromov2017questions}, this is one of the ``standard conjectures'' for scalar curvature and represents the most basic question one may ask about scalar curvature on manifolds with ``large'' fundamental groups. 

The Gauss--Bonnet theorem and classification of surfaces shows that Conjecture \ref{conj:asphericalPSC} holds for surfaces. For three-manifolds, Conjecture \ref{conj:asphericalPSC} was proven by Schoen and Yau \cite{SY:3d-torus} (using minimal surface methods) and Gromov and Lawson \cite{GL:complete} (using Dirac operator methods). In fact, by combining a mild generalization of this result with Perelman's resolution of the Poincar\'e conjecture one may obtain a complete classification of closed three-manifolds admitting Riemannian metrics with positive scalar curvature: 
Indeed, the resolution of the Poincar\'e conjecture implies that any orientable $3$-manifold has the prime decomposition
\[
M = S^3/\Gamma_1 \#\dots \# S^3/\Gamma_a \# b(S^2\times S^1) \# K_1 \# \cdots \# K_c,
\]
where each $\Gamma_1,\dots,\Gamma_a$ is a discrete subgroup of $SO(4)$ acting freely on $S^3$ and each $K_1,\dots,K_c$ is an aspherical manifold. Ruling out aspherical factors we obtain:
\begin{theorem}\label{theo:psc-3d}
A closed orientable\footnote{For the non-orientable case see \cite{LiZhang:cover}.} $3$-manifold $M$ admits a Riemannian metric of positive scalar curvature if and only if $M = S^3/\Gamma_1 \#\dots \# S^3/\Gamma_a \# b(S^2\times S^1) $. 
\end{theorem}
We emphasize that Perelman's work on the Ricci flow \cite{perelman:entropy,perelman:surgery} also proves this result directly.

In higher dimensions, Conjecture \ref{conj:asphericalPSC} has been proven for the following special classes of aspherical manifolds. 
\begin{enumerate}
\item Gromov and Lawson and Cecchini and Schick have proved \cite{GL:complete,CecchiniSchick} that Conjecture \ref{conj:asphericalPSC} holds for manifolds that admit metrics of non-positive sectional curvature (aspherical by the Cartan--Hadamard theorem). 
\item Rosenberg proved \cite{Rosenberg1983C*algebra} that if $M$ is a closed aspherical manifold and $\pi= \pi_1(M)$ satisfies the ``strong Novikov conjecture'' then $M$ does not admit a metric of positive scalar curvature. The strong Novikov conjecture is known\footnote{Sapir constructed \cite{Sapir:higman} closed aspherical $M^4$ whose fundamental groups are not known to satisfy the strong Novikov conjecture.} to hold for many groups (see e.g.\  \cite{Yu1998novikov}). 
\end{enumerate}
With Li, we were recently able to use minimal surface methods  to resolve  Conjecture \ref{conj:asphericalPSC} for four- and five-dimensional manifolds. Gromov simultaneously generalized our four-dimensional result to five dimensions.
\begin{theorem}[{\cite[Theorem 2]{ChodoshLi2020generalized}}, {\cite[Main Theorem]{Gromov2020metrics}}]\label{theo:asph45}
If $M^n$ is a closed aspherical manifold with $n \in \{2,3,4,5\}$ then it does not admit a Riemannian metric of positive scalar curvature. 
\end{theorem}
There are aspherical $4$-manifolds $M^4$ with $H_{3}(M;\ZZ) = 0$ (cf.\ \cite{RT}), so the minimization technique used in Theorem \ref{thm:geroch} does not directly apply in general (cf.\ \cite[Theorem F]{wang2019thesis}). Instead, we work on the universal cover. We give some indication of the proof of Theorem \ref{theo:asph45}, from \cite{ChodoshLi2020generalized} by revisiting the case of aspherical three manifolds. 

\begin{proof}[Sketch of the proof for $n=3$]
By scaling, assume that $(M,g)$ has $R\geq 1$. Thus, so does the universal cover $(\tilde M,\tilde g)$. Consider an annular region $W : = B_{r+L}(p) \setminus B_r(p)$. Theorem \ref{theo:mu-bubble-annulus} implies that there is a $\mu$-bubble in $W$ we find $\Sigma = \partial K$ with $K$ compact and $B_r(p) \subset K$ so that the induced metric on $\Sigma$ behaves as if it has strictly positive Gaussian curvature. In particular, $\Sigma$ has bounded diameter. Without loss of generality we assume there are no compact components of $\tilde M\setminus K$ (otherwise we can add them to $K$ without change). Suppose for contradiction that there is more than one component of $\Sigma$. Then, any component $\Sigma'$ separates $\tilde M$ into two non-compact sets, so $[\Sigma']\neq 0 \in H_2(\tilde M;\ZZ)$. Since $M$ is aspherical, $H_2(\tilde M;\ZZ) = 0$. This contradiction proves $\Sigma$ is connected. 

Since $(\tilde M,\tilde g)$ is contractible and the universal cover of a closed manifold, it's uniformly contractible (cf.\ \cite[Proposition 10]{ChodoshLi2020generalized}). This contradicts the fact that $\Sigma$ has bounded diameter  but the unique compact region it bounds $K$ has arbitrarily large diameter for $r$ large. 
\end{proof}

We briefly comment on the proof of Theorem \ref{theo:asph45} in higher dimensions. The first insight is that it's possible to ``reduce the dimension of the proof'' by one by constructing a curve linked with a far away codimension two submanifold. The argument above is then applied on an area-minimizing hypersurface bounded by the codimension two submanifold. This essentially handles the four-dimensional case. In five dimensions,  the resulting $\mu$-bubble will be three-dimensional. Three-dimensional Riemannian manifolds with scalar curvature $R\geq 1$ do not have diameter bounds so one must ``slice and dice'' as discussed in the sequel. 

\section{The geometry of positive scalar curvature.} \label{sec:PSC-geo}
Observing that for any closed $M^{n-2}$, the product $M^{n-2} \times S^2$ admits  positive scalar curvature (take a product metric where the sphere has very small radius), we obtain the vague sense that a manifold of positive scalar curvature is ``small'' in at least two dimensions. This notion can be made precise in various ways, but the most relevant one for us is the following notion.
\begin{definition}
We say that a metric space $(X,d)$ has $k$-Urysohn width $\leq W$ if there exists a $k$-dimensional simplicial complex $Y$ and a continuous map $f: X\to Y$ with $\diam(f^{-1}(y)) \leq W$ for all $y \in Y$.
\end{definition}
The Bonnet--Myers theorem says that if $(M^n,g)$ has $\Ric \geq n-1$ then $(M,g)$ has $0$-Urysohn width $\leq \pi$. The following is a well-known conjecture of Gromov:
\begin{conjecture}[{\cite[Conjecture 34]{gromov2017questions}}]\label{conj:urysohn}
If $(M^n,g)$ is closed with scalar curvature $R\geq 1$ then $(M,g)$ has $(n-2)$-Urysohn width $\leq W=W(n)$. 
\end{conjecture}
Alternatively, one may ask if the universal cover $(\tilde M,\tilde g)$ has finite $(n-2)$-Urysohn width. 

Gromov and Lawson used the inductive descent technique applied to minimal disks to prove that a simply connected complete Riemannian $3$-manifold $(M^3,g)$ with scalar curvature $R \geq 1$ has finite $1$-Urysohn width. More recently, Conjecture \ref{conj:urysohn} has been proven by  Liokumovich and Maximo \cite{LiokumovichMaximo2020waist} in a strong form. An alternative proof may be given using the ``slice and dice'' argument from \cite{ChodoshLi2020generalized}. For simplicity, we assume $M$ is simply connected\footnote{A proof using slice and dice for non-simply connected manifolds is given in  \cite[Theorem 2.2.5]{Balitskiy}.}   to give an idea of the ``slice and dice'' argument:
\begin{proof}[Sketch of a proof of Conjecture \ref{conj:urysohn} for simply connected $3$-manifolds]
Let $(M,g)$ be complete and simply connected with $R\geq 1$. We start with a tiny ball $B_\varepsilon(p)$ and consider the band $W_1 : = B_{\varepsilon+L}(p)\setminus B_\varepsilon(p)$. Theorem \ref{theo:mu-bubble-annulus} allows us to find a $\mu$-bubble $\Sigma_1$ separating $W_1$.  Each component of $\Sigma_1 = \partial K_1$ has uniformly bounded diameter. We can repeat this argument to find a $\mu$-bubble $\Sigma_2 = \partial(K_1\cup K_2)$ \emph{outside} of $\Sigma_1$ at a distance $\leq L$. We continue this process (possibly infinitely many times) until we have exhausted $(M,g)$

By simple connectivity, each component $K_{i+1}'$ of $K_{i+1}$ intersects exactly one component of $\Sigma_i$. Each point in $K_{i+1}$ is a distance $\leq L$ from $\Sigma_i$ (by construction), so the diameter of $K_{i+1}'$ is bounded. This shows that  $(M,g)$ coarsely resembles a tree with a vertex at each component of $K_i$ with edges along  components of $\Sigma_i$. 
\end{proof}
When $(M,g)$ is not simply connected, the previous argument might have difficulty when $K_{i+1}'$ intersects multiple components of $\Sigma_i$. In this case, all we know is that $p\in K_{i+1}$ has distance $\leq L$ from $\Sigma_i$ so we cannot conclude that $K_{i+1}'$ has bounded diameter. Instead, in \cite{ChodoshLi2020generalized} one obtains an appropriate decomposition of $(M,g)$ by first cutting along a generating set of $H_2(M;\ZZ)$ by area-minimizing surfaces (this is the ``slice''). Then, one uses free boundary $\mu$-bubbles (the ``dice'') as above. This decomposes $(M,g)$ into regions of bounded diameter and overlap.

Conjecture \ref{conj:urysohn} is widely open in higher dimensions. Partial results are known, particularly for certain $\pi = \pi_1(M)$ (related to results on the strong Novikov conjecture) by works of  Bolotov and Dranishnikov \cite{Bolotov,BolotovDranishnikov}. With Li and Liokumovich we  generalized techniques from \cite{ChodoshLi2020generalized,ABG} to prove a variant of Conjecture \ref{conj:urysohn}:
\begin{theorem}[{\cite[Theorem 4]{CLL:suff.conn.psc}}]
For $n \in\{4,5\}$, suppose that $(M,g)$ is a closed $n$-dimensional Riemannian manifold with positive scalar curvature. If $n=4$ assume that $\pi_2(M) = 0$ and if $n=5$ assume that $\pi_2(M) = \pi_3(M) = 0$. Then the universal cover $(\tilde M,\tilde g)$ has finite $1$-Urysohn width. 
\end{theorem}
As $T^2 \times S^2$ has $R\geq 1$ but the universal cover $\RR^2\times S^2$ has infinite $1$-Urysohn width, the condition on the higher homotopy groups cannot be removed. Combining this argument with arguments from geometric group theory and topology we obtained the following partial classification result:
\begin{corollary}[{\cite[Theorem 1]{CLL:suff.conn.psc}}]\label{coro:suff-conn}
For $n \in\{4,5\}$, suppose that $(M,g)$ is a closed $n$-dimensional Riemannian manifold with positive scalar curvature. If $n=4$ assume that $\pi_2(M) = 0$ and if $n=5$ assume that $\pi_2(M) = \pi_3(M) = 0$.Then some finite cover $\hat M$ is homotopy equivalent to $S^n$ or a connected sum of $S^{n-1}\times S^1$ factors. 
\end{corollary} 

\section{Non-compact manifolds with positive scalar curvature.} 
The $\mu$-bubble technique has been particularly effective in the study of complete\footnote{Note that any non-compact manifold admits an incomplete metric of positive scalar curvature so we must assume completeness here (cf.\ \cite{Rosenberg:PSC.progress}). } Riemannian manifolds with uniformly positive scalar curvature $R\geq 1$. In particular, we have the following result (the proof is like the $n=3$ proof of Conjecture \ref{conj:urysohn} above):
\begin{theorem}\label{theo:exhaustion}
For $n \leq 7$, if $(M^n,g)$ is a complete non-compact Riemannian manifold with scalar curvature $R\geq 1$ then there's an exhaustion $\Omega_1\subset \Omega_2\subset \Omega_3\subset \dots$ by compact sets so that each $\partial\Omega_i$ is smooth and has positive scalar curvature in the spectral sense. 
\end{theorem}
Thus, classification results for closed manifolds admitting positive scalar curvature can be used to control the topology of $\partial\Omega_i$. For example, J.\ Wang has used this to prove the following:
\begin{theorem}[{\cite[Theorem 1.2]{wang:strictPSC}}]
Suppose that $(M^3,g)$ is an oriented, complete Riemaniann manifold with $R\geq 1$. Then $M$ is a possibly infinite connected sum of spherical $3$-manifolds and $S^1\times S^2$ factors. 
\end{theorem}
This was previously obtained via Ricci flow in \cite{BBM:open-psc} with an additional assumption that $g$ has bounded geometry (bounded curvature and injectivity radius). In higher dimensions such a classification is not known, but with Maximo and Mukherjee we obtained:
\begin{theorem}[{\cite[Theorem 1.2]{CMM}}]\label{thm:4d-contr}
For $W$ a compact contractible smooth four-manifold with boundary, let $M$ be the interior of $W$. If $M$ admits a complete Riemannian metric with scalar curvature $R\geq 1$ then $W$ is homeomorphic to the four-ball. 
\end{theorem}
A version of this for sufficiently connected, five-dimensional contractible manifolds was recently obtained by Sweeney \cite{Sweeney}. We also note that  C.\ Li and B.\ Zhang have used tools from topology to obtain interesting results on stable minimal hypersurfaces in certain four-manifolds \cite{LiZhang:topMin}.

\begin{proof}[Sketch of the proof]
Let $\Omega_1\subset \Omega_2\subset \dots$ be the exhaustion from Theorem \ref{theo:exhaustion}. For $i$ large, there's a non-zero degree map $\partial\Omega_i \to \partial W$. One may generalize Theorem \ref{theo:psc-3d} to a mapping version (cf.\ \cite[Proposition 2.2]{CMM}) and use this to conclude $W$ is a connected sum of spherical summands and $S^2\times S^1$ factors. Since $W$ is contractible, Lefschetz duality shows that $\partial W$ is a homology sphere $H_*(\partial W,\ZZ) = H_*(S^3;\ZZ)$. Thus, there are no $S^2\times S^1$ factors. Furthermore, by the study of the fundamental groups of spherical space forms, we find that $\partial W$ is either $S^3$ or else a connected sums of Poincar\'e homology spheres. 

We now combine several deep results in topology. By combining the Heegard Floer $d$-invariant \cite{OS:d} and work of Taubes \cite{Taubes} we find that $\partial W = S^3$. The assertion then follows from the work of Freedman \cite{Freedman}. 
\end{proof}

Using an argument similar to Theorem \ref{thm:4d-contr} and an invariant called end Floer homology, we also proved that scalar curvature can detect certain smooth structures:
\begin{theorem}[{\cite[Theorem 1.3]{CMM}}]
There exist smooth structures on $\bR^4$ that do not admit complete Riemannian metrics with scalar curvature $R\geq 1$. 
\end{theorem}
A natural question (cf.\ \cite{CWY}) is if the only smooth structure $\bR^4$ with such a metric is the standard one.

\subsection{Non-strictly positive scalar curvature} The $\mu$-bubble exhaustion property depends crucially on strict positivity of scalar curvature and does not seem to generalize to the case of positive scalar curvature decaying rapidly at infinity. For example, there is no classification of three-manifolds admitting complete metrics with scalar curvature $R>0$ as evidenced by the following conjectures:
\begin{conjecture}\label{conj:contract}
If $(M^3,g)$ is a contractible $3$-manifold with a complete metric with scalar curvature  $R>0$ then $M$ is diffeomorphic to $\bR^3$. \end{conjecture}

\begin{conjecture}\label{conj:handlebodies}
If $(M^3,g)$ is the interior of a genus $\gamma$ handlebody with a complete metric with scalar curvature  $R>0$ then\footnote{Note that the interior of a genus-one handlebody admits a complete metric with scalar curvature $R>0$ (cross a circle with a positively curved metric on the plane).} $\gamma \leq 1$. \end{conjecture}

Partial progress on Conjecture \ref{conj:contract} has been made by J.\ Wang (cf.\ \cite{CWY}) using minimal surfaces. For example: 
\begin{theorem}[\cite{Wang:contractible1}]\label{thm:whitehead}
There is no complete metric with $R>0$ on the Whitehead manifold.
\end{theorem}
\begin{proof}[Sketch of the proof]
The Whitehead manifold may be exhausted by an increasing union of solid tori $\Omega_1\subset \Omega_2 \subset \dots$ with the property that if $D\subset \Omega_{i+1}$ is a disk with $\partial D \subset \partial\Omega_{i+1}$ a meridional curve, then $D$ intersects the core of $\Omega_i$ in at least two points. 

Assume for contradiction that there is a complete metric of positive scalar curvature on the Whitehead manifold. Let $\gamma_i \subset \partial\Omega_i$ be meridional and minimize area among disks bounded by $\gamma_i$ to find $D_i\subset M$ with $\partial D_i = \gamma_i$. Note that $D_i$ must intersect $\Omega_{i-1}$ in at least two meridional curves, which then implies it intersects $\Omega_{i-2}$ in at least four meridional curves and so on. Curvature estimates for stable minimal surfaces imply that some subsequence of $D_k$ limits in $C^\infty_\textrm{loc}$ to a complete stable minimal surface $\Sigma$ so that for $I$ fixed, $\Sigma$ intersects $\partial \Omega_I$ in at least one meridional curve, and thus  $\Sigma \cap \Omega_1$ has at least $2^I$ disjoint components. 

On the other hand, a crucial observation in \cite{Wang:contractible1} is that a complete stable minimal surface $\Sigma$ in $R>0$ has 
\[
\int_\Sigma R \leq 2\pi.
\]
This is proven by combining  the stability inequality with the Cohn--Vossen inequality for total curvature. Since $R\geq R_0>0$ on $\Omega_1$ by compactness, we thus find that \[\area(\Sigma \cap \Omega_1) \leq C\] with $C$ independent of $\Sigma,I$. On the other hand, the monotonicity formula for minimal surfaces and the previous paragraph implies that \[\area(\Sigma\cap \Omega_1)\geq c 2^I\] for $c$ independent of $\Sigma,I$. Taking $I\to \infty$ this is a contradiction.
\end{proof}

 With Y.\ Lai and K.\ Xu, we recently proved  Conjectures \ref{conj:contract}  and \ref{conj:handlebodies}, assuming bounded geometry, e.g.:
\begin{theorem}[\cite{CLX}]
If $M^3$ is a contractible $3$-manifold that admits a complete metric of positive scalar curvature $R>0$, bounded sectional curvature, and injectivity radius bounded below, then $M=\bR^3$. 
\end{theorem}
 \begin{proof}[Sketch of the proof]
We use inverse mean curvature flow to replace $\mu$-bubbles. Recall that a smooth surface flows by inverse mean curvature flow if it has outward speed $\frac{1}{H}$. Such a flow is likely to develop singularities, but a suitable notion of weak inverse mean curvature flow was developed by Huisken and Ilmanen en route to their proof of the Penrose inequality \cite{HI}. A key feature of an inverse mean curvature flow $\Sigma_t \subset (M^3,g)$ in a manifold with non-negative scalar curvature is the Geroch monotonicity formula 
\begin{equation}\label{eq:IMCFgeroch}
 \frac{d}{dt}  \int_{\Sigma_t} H^2   \leq 4\pi \chi(\Sigma_t) - \frac 12 \int_{\Sigma_t} H^2
\end{equation}
and (remarkably) this continues to hold for a weak flow. 

Suppose that $(M^3,g)$ is contractible with a complete metric of positive scalar curvature and bounded geometry. If there exists a weak inverse mean curvature flow for $t \in (0,\infty)$, then  \eqref{eq:IMCFgeroch} implies that $\chi(\Sigma_t) \geq 0$ for infinitely many $t$ (otherwise, $\int_{\Sigma_t}H^2$ would become negative for sufficiently large $t$). Thus, $M$ is exhausted by handlebodies of genus $\leq 1$. If $M\neq \bR^3$, then the handlebodies must be solid tori, nested like in the Whitehead manifold.  Theorem \ref{thm:whitehead} generalizes to cover this case, finishing the proof.

A crucial issue here is that the weak inverse mean curvature flow may want to ``jump'' to infinity in finite time. This case is more delicate. Briefly, Xu's recent existence results for weak inverse mean curvature flow  \cite{Xu:IMCF,Xu:outer} are combined with the bounded geometry assumption to ``track the jumping surface to infinity,''  which yields a similar torus exhaustion, finishing the proof as before. 
 \end{proof}
 
Our work suggests the following problem  (see \cite[\S 1.2]{CLX}):
\begin{conjecture}
Assume that $M^3$ has an exhaustion by solid tori and admits a complete metric with scalar curvature $R>0$ and bounded geometry. Then $M^3$ is the interior of a genus-one handlebody. 
\end{conjecture}
 
In higher dimensions, very little is known concerning complete metrics with scalar curvature $R>0$. We briefly mention the following result obtained with Li  (see also \cite{LUY:liouville}), again using $\mu$-bubbles:
\begin{theorem}[{\cite[Theorem 3]{ChodoshLi2020generalized}}]\label{theo:torus-noncompact}
For $n\leq 7$, if $M^n$ is a possibly non-compact smooth manifold then $T^n\# M^n$ does not admit a complete Riemannian metric with scalar curvature $R>0$. 
\end{theorem}
Along with \cite{SY:conf-flat,LUY:liouville}, this provided the missing ingredient in the Schoen--Yau Liouville theorem \cite{SY:conf-flat} for locally conformally flat manifolds, with strong consequences for higher homotopy groups of locally conformally flat manifolds with nonnegative scalar curvature. 

S.\ Chen, J.\ Chu, and J.\ Zhu have recently proved results \cite{ChenChuZhu,ChenChuZhu2} that combine and generalize Theorems \ref{theo:asph45} and \ref{theo:torus-noncompact}. Among other things, they proved that for $n\in \{3,4,5\}$, punctured aspherical $n$-manifolds do not admit complete metrics with scalar curvature $R>0$.

\section{Other results.} Here are several related results (cf.\ \cite{Rosenberg:PSC.progress,gromov2019lectures,ChodoshLi:survey,Chodosh:minscal,Stolz:survey}).  Gromov and Lawson \cite{GL:simply-connected-PSC} as well as Schoen and Yau \cite{SY:manuscripta} have shown that surgeries of codimension $\geq 3$ can be done so as to preserve positivity of scalar curvature. Using this, Gromov and Lawson  and Stolz  were able to give a complete classification  of simply connected closed manifolds of dimension $\geq 5$ that admit metrics of positive scalar curvature. 
\begin{theorem}[\cite{GL:simply-connected-PSC,Stolz:simply.conn}]\label{theo:simply-PSC}
For $n\geq 5$ consider $M^n$ a closed, simply connected, smooth manifold. 
\begin{enumerate}
\item If $M$ is non-spin then it admits a metric of positive scalar curvature.
\item If $M$ is spin then it admits a metric of positive scalar curvature if and only if $\alpha(M) = 0$. 
\end{enumerate}
\end{theorem}
Here, $\alpha(M) \in KO(S^n)$ is a topological invariant that vanishes for a spin manifold that admits positive scalar curvature (this is due to Hitchin \cite{Hichin:harmonic-spinor} who generalized the classical Lichnerowicz obstruction \cite{Lichnerowicz}). For general $\pi_1(M)$, such a result is not known, cf.\ \cite{Schick1998counterexample}. 

The situation for four-manifolds (even simply connected) is much more complicated since in addition to the minimal surface and Dirac obstructions there are also the Seiberg--Witten invariants. We do not discuss this here except to note that there exist examples showing that Theorem \ref{theo:simply-PSC} fails for four-manifolds \cite[Counterexample 1.13]{Rosenberg:PSC.progress}. On the other hand, there is some hope that the Ricci flow could lead to (partial) classification results for closed four-dimensional admitting Riemannian metrics of positive scalar curvature. Motivated by this picture, Bamler, Li, and Mantoulidis used \cite{BLM} minimal surface techniques somewhat similar to the ``slice'' argument in \cite{ChodoshLi2020generalized} but combined with techniques used to ``cap off'' cuts by stable minimal hypersurfaces from \cite{MantSch,LiMantoulidis:spect} to prove:
\begin{theorem}
A closed oriented smooth four-manifold $M$ that admits a metric of positive scalar curvature can be obtained via $0,1$-surgeries\footnote{The $0$-surgeries (connected sums) may occur at orbifold points.} from a possibly disconnected, closed, oriented four-orbifold $M'$ with isolated singularities also admitting positive scalar curvature and with $b_1(M') = 0$. 
\end{theorem}
They ask if it is possible to arrange the stronger condition $\pi_1^\textrm{orb}(M') = 0$. 

In higher dimensions, R\"ade has combined minimal surface techniques with surgery theory to prove  the following $S^1$-stability result as conjectured by Rosenberg. Combined with the generic regularity results as discussed in Section \ref{sec:generic} we have:
\begin{theorem}[\cite{Rade}, cf.\ \cite{CRZ}]\label{theo:S1-stab}
If $M^n$ is a closed smooth manifold with $n \in \{2,\dots,10\}\setminus\{4\}$ then $M$ admits a metric of positive scalar curvature if and only if $M \times S^1$ does. 
\end{theorem}

The $n=2$ case follows directly from Theorem \ref{theo:psc-3d}. For $n=3$, if $M^3\times S^1$ admits positive scalar curvature but $M^3$ does not then $M^3$ has an aspherical factor $K$ in its prime decomposition. Then $M^3\times S^1$ admits a degree one map to the aspherical four-manifold $K\times S^1$.  This contradicts the mapping version of Theorem \ref{theo:asph45} (cf.\ \cite{CLL:suff.conn.psc}). See also \cite{Agol:MO} for a proof using deep results in three-manifold topology and Schoen--Yau inductive descent. Surprisingly, Theorem \ref{theo:S1-stab} fails for $n=4$ (since the Seiberg--Witten obstructions for $M$ can ``dissolve'' after taking the $S^1$-product), see \cite[Remark 1.25]{Rosenberg:PSC.progress}. For $n \in \{5,\dots,10\}$, the strategy from \cite{Rade} is to find $\Sigma$ minimizing area in the homology class $[M\times\{0\}] \in H_n(M\times S^1;\ZZ)$ which then admits a metric of positive scalar curvature if $M\times S^1$ does. When the dimension of $M$ is $\geq 5$, surgery theory then allows one to promote this to a metric of positive scalar curvature on $M$. If $n\leq 6$ then $\Sigma$ will not have any singularities (cf.\ Theorem \ref{theo:minimize}). When $n\in \{7,8,9,10\}$, the results in Section \ref{sec:generic} show that the positive scalar curvature metric on $M\times S^1$ may be perturbed so that the resulting minimizer $\Sigma$ is smooth. The proof is then completed as before. 

Finally, we briefly review work relating scalar curvature and geometric conditions such as systolic inequalities. For example, if we write $\sys_2(M,g)$ for the least area non-zero element of $\pi_2(M)$ a well-known result of Bray, Brendle, and Neves gives that if $(M^3,g)$ has $\pi_2\neq 0$ and $R\geq 2$ then 
\[
\sys_2(M,g)\leq 4\pi
\] with equality only for a quotient of the cylinder $\bR \times \bS^2$. This has been recently generalized by K.\ Xu using inverse mean curvature flow who proved \cite{Xu:systole} that if $(M,g)$ is not topologically covered by $\bR \times \bS^2$ then 
\[
\sys_2(M,g) \leq 2.8\pi.
\]
We also note the recent resolution of the Horowitz--Myers conjecture by Brendle and Hung \cite{BrendleHung:HM} who (roughly speaking) prove a systolic inequality 
\[
2\sigma^n \inf_{\partial M} (H-(n-1)) \leq \left( \frac{4\pi}{n}\right)^n
\]
for the length $\sigma$ of the shortest non-contractible curve on the boundary of $(B^2\times T^{n-1},g)$, under the assumption that the scalar curvature satisfies $R\geq -n(n-1)$.

\section{Other curvature conditions}

Scalar curvature is not the only curvature condition that may be studied via minimal surfaces. For example Y.\ Shen and R.\ Ye introduced the bi-Ricci curvature, defined as
\[
\biRic(u,v) = \Ric(u,u) + \Ric(v,v) - \sec(u,v)
\]
for all orthonormal vectors $u,v$. One may check that in three-dimensions the biRicci curvature equals the scalar curvature (up to a multiplicative constant) but in higher dimensions positivity of biRicci implies positivity of scalar curvature (but not vice versa). The following is a key observation of Shen and Ye:
\begin{theorem}[\cite{ShenYe,SY:general}]\label{thm:SY-biRic-stab}
Suppose that $\Sigma^n\subset (M^{n+1},g)$ is a two-sided stable minimal hypersurface in a manifold with $\biRic \geq 1$. Write $\lambda_{\Ric,\Sigma}$ for the smallest eigenvalue of the Ricci curvature of the induced metric on $\Sigma$. Then the Ricci curvature of the induced metric $\Sigma$ is strictly positive in the spectral sense. 
\end{theorem}
\begin{proof}
Let $e_1,\dots,e_n,\nu$ denote an orthonormal basis for $T_pM$ with $\lambda_{\Ric,\Sigma} = \Ric_\Sigma(e_1,e_1)$. Using the Gauss equations we obtain
\begin{align*}
\lambda_{\Ric,\Sigma} & = \sum_{j=2}^n \sec_\Sigma(e_1,e_j)\\ &  = \sum_{j=2}^n (\sec_M(e_1,e_j) + A(e_1,e_1)A(e_j,e_j) - A(e_1,e_j)^2) \\&  \geq \biRic(e_1,\nu) - \Ric_M(\nu,\nu) - |A|^2
\end{align*}
where we used $\tr A = 0$ in the last step. Using $\biRic(e_1,\nu)\geq 1$, stability of $\Sigma$ as in \eqref{eq:stable-min} may be rewritten as
\[
\int_\Sigma \varphi^2 \leq \int_{\Sigma}|\nabla \varphi|^2 + \lambda_{\Ric,\Sigma}\varphi^2 .
\]
This completes the proof.
\end{proof}
One may expect that the Bonnet--Myers theorem holds for spectral Ricci curvature. Indeed, Shen and Ye were able to prove the following diameter bound (generalizing the work of Schoen and Yau \cite{SY:condensation}):
\begin{theorem}[\cite{ShenYe,SY:general}]
Suppose that $\Sigma^n\subset (M^{n+1},g)$ is a two-sided stable minimal hypersurface in a manifold with $\biRic \geq 1$. Then for $n+1 \leq 5$ each component of $\Sigma$ has diameter $\leq c(n)$. 
\end{theorem}
 Surprisingly, K.\ Xu constructed examples \cite{Xu:counterex} where the diameter bound fails for $n+1> 5$. 
 
 The bi-Ricci curvature was generalized by Brendle, Hirsch, and Johne to a class of curvatures interpolating between Ricci and scalar curvature (called the $m$-intermediate curvatures) defined by 
 \[
 \mathcal{C}_m(e_1,\dots,e_m) : = \sum_{p=1}^m \sum_{q=p+1}^n \sec(e_p,e_q).
 \]
 (Note that $\mathcal{C}_2$ is the bi-Ricci curvature.) Using a version of inductive descent (cf.\ Section \ref{sec:induct})  they proved:
\begin{theorem}[{\cite[Corollary 1.6]{BHJ}}]
For $n\leq 7$ and $1\leq m\leq n-1$, let $M^{n-m}$ be a closed manifold. Then $M^{n-m}\times T^m$ does not admit a metric with $\mathcal{C}_m > 0$. 
\end{theorem}
These techniques have allowed for many results concerning minimal surfaces in positive scalar curvature to be extended to higher dimensional manifolds with positive $\mathcal{C}_m$ (for appropriate $m$), see e.g. \cite{MWY:intermediateasph}.

\section{The (stable) Bernstein problem.} 
We now discuss the Bernstein problem and its modern variants. As we'll explain, major progress in this area came from the discovery of links with the previous sections.

We will say that $M^n\subset \bR^{n+1}$ is an \emph{entire minimal graph} if $M$ is a minimal hypersurface and a graph of $u : \bR^n\to\bR$. One may easily check that this is equivalent to requiring that $u$ satisfies the minimal surface equation
\begin{equation}\label{eq:MSE}
\textrm{div} \left( \frac{\nabla u}{\sqrt{1+|\nabla u|^2}} \right) = 0.
\end{equation}
This may be regarded as a nonlinear analogue of the Laplace equation $\Delta u = 0$. In 1914, Bernstein proved \cite{Bernstein} that an entire solution to \eqref{eq:MSE} on $\bR^2$ is affine. This can be regarded as a non-linear version of Liouville's theorem and is particularly remarkable since no boundedness assumption on $u$ is required. It is natural to ask about higher dimensions. A remarkable series of works in geometric measure theory resolved this in the 1960s:
\begin{theorem}\label{theo:bernstein}
For $n \leq 7$, a minimal graph $u : \bR^n\to \bR$ is affine. On the other hand, for $n\geq 8$, there are non-flat minimal graphs on $\bR^{n}$.
\end{theorem}
For later reference, we briefly outline the steps of the proof:
\begin{enumerate}
\item If $M\subset \bR^{n+1}$ is a minimal graph then a calibration argument proves that it's area-minimizing (and thus stable). 
\item Using ``monotonicity of area ratios'' for minimal surfaces and compactness for least area hypersurfaces, Flemming  \cite{Fleming} proved that if $u : \bR^{n}\to \bR$ is a non-affine solution to the minimal surface equation then there's a non-flat area-minimizing hypercone $C\subset \bR^{n+1}$. 
\item De Giorgi proved \cite{DeGiorgi:bernstein} that the hypercone actually splits $C=C'\times \bR$ and thus there's a non-flat area-minimizing hypercone $C\subset \bR^{n}$. 
\item Generalizing work \cite{Almgren:bernstein} of Almgren, Simons used the stability inequality to prove \cite{Simons:minimalVarieties} that there are no non-flat area-minimizing hypercones in $\bR^n$ for $n\leq 7$. 
\item Bombieri, De Giorgi, and Giusti \cite{BombieriDeGiorgiGiusti} constructed a non-affine graph over $\bR^8$. 
\end{enumerate}

We note that the non-existence non-flat of area-minimizing hypercones (as in the fourth point above) is precisely the reason for the regularity of area minimizers in low dimensions (Theorem \ref{theo:minimize}). 

A natural generalization of the Bernstein problem is to drop the graphical condition. If we retain the area-minimizing assumption then the problem was completely resolved by the previously mentioned work:
\begin{theorem}
If $\Sigma^n\subset \bR^{n+1}$ is a complete area-minimizing hypersurface and $n+1 \leq 7$ then $\Sigma$ is flat. For $n+1\geq 8$ there are non-flat examples. 
\end{theorem}
The shift in dimensions is because De Giorgi's splitting result does not hold without graphicality. 

By replacing ``area-minimizing'' with ``stable'' we obtain the \emph{stable Bernstein problem}: If $\Sigma^n\to \RR^{n+1}$ is a complete two-sided stable minimal immersion, is it flat? One important reason to study the stable Bernstein problem is that it is equivalent to \emph{a priori} curvature estimates for stable minimal surfaces in Riemannian manifolds. 

In $\bR^3$ this was resolved in 1980 by Fischer-Colbrie and Schoen \cite{Fischer-Colbrie-Schoen}, do Carmo and Peng \cite{doCarmoPeng}, and Pogorelov \cite{Pogorelov}. These proofs leverage the two-dimensionality of $M$ in various strong ways (uniformization, negative curvature of a minimal surface in $\bR^3$).

In higher dimensions, well-known work of Schoen, Simon, and Yau, Schoen and Simon, and a recent breakthrough by Bellettini  resolves the stable Bernstein problem assuming volume growth bounds. (See also the recent works \cite{CFRFS,FSS} for stable Bernstein results for other geometric problems.)
\begin{theorem}[\cite{SSY,SchoenSimon,Bellettini}]\label{theo:SSY-etc}
Suppose that $\varphi : \Sigma^n \to \bR^{n+1}$ is a complete two-sided stable minimal immersion that satisfies extrinsic volume growth estimates $\area(\varphi(\Sigma) \cap B_R) \leq O(R^n)$ for all $R>0$. If $n+1\leq 7$ then $\Sigma$ is flat. With only intrinsic volume growth estimates $\area(B_R^\Sigma) \leq O(R^n)$, the same conclusion holds for $n+1\leq 6$. 
\end{theorem}

Thus, the resolution of the stable Bernstein problem becomes a question of a more geometric nature: Can the stability inequality can be used to derive volume growth bounds for $\Sigma^n\to\bR^{n+1}$? This geometric problem has recently seen considerable activity and is currently almost completely resolved (only the case $M^6\to\bR^7$ remains unsolved). With C.\ Li, we solved the $M^3\to\bR^4$ case \cite{ChodoshLi:stableR4}. Later we discovered a second proof \cite{ChodoshLi:anisoR4} and shortly after a third proof was discovered by Catino, Mastrolia, and Roncoroni \cite{CMR}. Our first proof was based on techniques related to nonnegative scalar curvature (cf.\ \cite{MunteanuWang,Stern}) while our second proof was based on techniques related to strictly positive scalar curvature. Finally, the Catino--Mastrolia--Roncoroni proof was based on Bakry--\'Emery Ricci curvature. We describe here the proof from \cite{ChodoshLi:anisoR4} based on strict positivity of scalar curvature. 
\begin{theorem}\label{thm:R4stab}
If $\Sigma^3\to\bR^4$ is a complete two-sided stable minimal immersion then it's flat. 
\end{theorem}
\begin{proof}[Sketch of the proof]
The stability inequality \eqref{eq:stable-min} in Euclidean space reads
\[
\int_\Sigma |A|^2 \varphi^2 \leq \int_\Sigma |\nabla \varphi|^2.
\]
Using the Gauss equations we have $|A|^2 = -R_\Sigma$. Thus, the stability inequality is the same thing as spectral nonnegativity of $-\Delta_\Sigma + R_\Sigma$. To promote this to spectral positivity we use an old observation due to Gulliver and Lawson \cite{GulliverLawson}. Letting $g$ denote the induced metric on $\Sigma$ and $r : \Sigma\to\RR$ the restriction of the ambient distance $r(x) = |x|$. Then, we consider the conformally equivalent metric $\tilde g = r^{-2} g$. (Motivation for considering $\tilde g$ may be seen from the example of $\bR^n\subset \bR^{n+1}$ so that $\tilde g$ becomes the product metric on $\bR \times \bS^{n-1}$; note that for $n\geq 3$ this metric has strictly positive scalar curvature.)

Conformally changing the nonnegative operator $-\Delta_\Sigma + R_\Sigma$ into $\tilde g$-quantities gives that $(\Sigma,\tilde g)$ satisfies
\[
-2 \tilde \Delta + \tilde R \geq c(n) > 0
\]
for $n\geq 3$. As such, we find that $(\Sigma,\tilde g)$ has \emph{uniformly positive} scalar curvature in a spectral sense. As such, we may apply a variant of the $\mu$-bubble exhaustion technique to prove that for any $\rho>0$, there's a compact region $\Omega$  so that $B_\rho^\Sigma(p) \subset \Omega$, $\partial\Omega \subset B_{C\rho}^\Sigma(p)$, and $\partial \Omega$ admits strictly positive scalar curvature in a spectral sense. 

A well-known result \cite{CSZ} (discussed further in the sequel) of H.-D.\ Cao, Y. Shen, and S. Zhu gives that $\Sigma$ has only one end. As such, it's possible to modify $\Omega$ so that $\partial\Omega$ is connected. Since $n=3$, $\partial\Omega$ is two-dimensional. Thus, we have (spectral positivity) of Gaussian curvature which implies that $\partial\Omega$ (with the metric induced by $\tilde g$) is a sphere of bounded diameter and bounded volume. After undoing the conformal change, we thus find that $\area(\partial\Omega)=O(\rho^2)$ (with the metric $g$ induced on $\Sigma$ by the immersion). Using that minimal surfaces satisfy the Euclidean isoperimetric inequality by work of Michael and Simon \cite{MSsob} (see also the recent breakthrough by Brendle \cite{Brendle:iso}) this yields the intrinsic volume growth estimate $\vol(\Omega) = O(\rho^3)$. By Theorem \ref{theo:SSY-etc}, this completes the proof.  
\end{proof}
Subsequently, with Li, Minter, and Stryker we discovered \cite{CLMS} that one may generalize this argument to  $\Sigma^4\to\bR^5$ by using bi-Ricci curvature in place of scalar curvature. In particular, a key observation is that although $|A|^2 = - R_\Sigma$ by the doubly traced Gauss equations, $|A|^2$ actually controls the full curvature tensor of $(\Sigma,g)$. For example, the Gauss equations give
\begin{align*}
\Ric_\Sigma(e_1,e_1) & = \sum_{i=2}^n \sec_\Sigma(e_1,e_i)\\ & = \sum_{i=1}^n A(e_1,e_1)A(e_i,e_i) - \sum_{i=1}^n A(e_1,e_i)^2 \\& \geq - \frac{n-1}{n} |A|^2
\end{align*}
where we used the AM-GM inequality and $\tr A = 0$ in the last step (this is closely related to the well-known ``improved Kato inequality''). Thus, stability of $\Sigma^n\to\bR^{n+1}$ implies that $(\Sigma,g)$ has nonnegative spectral \emph{Ricci} curvature in some sense. In hindsight, one may view the Cao--Shen--Zhu one-ended result used in the proof of Theorem \ref{thm:R4stab} as the combination of this observation with a result that (roughly stated) proves that a complete Riemannian manifold with spectrally nonnegative Ricci curvature that admits a Euclidean Sobolev inequality must have only one end of infinite area. 

One may repeat this kind of calculation for bi-Ricci curvature and then combine it with the Gulliver--Lawson conformal change to deduce that if $\Sigma^4\to\bR^5$ is a complete two-sided stable minimal immersion then the conformally changed metric $(\Sigma,\tilde g)$ has strictly positive bi-Ricci curvature in the spectral sense. In particular, the $\mu$-bubble technique in the proof of Theorem \ref{thm:R4stab} yields $\partial\Omega$ with spectrally positive Ricci curvature. The final step was then to prove that $\partial\Omega$ has uniformly bounded volume. This was achieved by proving a spectral Bishop--Gromov style volume comparison theorem using the isoperimetric surface approach introduced by Bray in his thesis \cite{Bray:thesis}. Combining these elements, the proof of the stable Bernstein theorem for $\Sigma^4\to\bR^5$ then follows in a similar way as in $\bR^4$. 

The spectral Bishop--Gromov volume comparison result (which we could only prove for $3$-manifolds) was subsequently generalized to all dimensions by Antonelli and Xu who proved:
\begin{theorem}[{\cite[Theorem 1]{AntonelliXu:BG}}]\label{thm:AXBG}
Suppose that $(M^n,g)$ has spectrally positive Ricci curvature $-\frac{n-1}{n-2} \Delta + \lambda_{\Ric} \geq (n-1)$. Then $\vol(M) \leq \vol(\bS^n)$. 
\end{theorem}
Shortly after this result appeared, Mazet was able to generalize the techniques described above by using Theorem \ref{thm:AXBG}, with a weighted notion of bi-Ricci curvature (originally introduced by Shen and Ye), along with a delicate choice of coefficients to resolve the stable Bernstein problem for $\Sigma^5\to\bR^6$. In sum, this yields
\begin{theorem}[\cite{Fischer-Colbrie-Schoen,doCarmoPeng,Pogorelov,ChodoshLi:stableR4,ChodoshLi:anisoR4,CMR,CLMS,Mazet}]\label{thm:stab-bern-alldim}
For  $n+1\leq 6$, if $\Sigma^n\to\bR^{n+1}$ is a complete two-sided stable minimal immersion then it's flat. 
\end{theorem}
This leaves only the case of $\Sigma^6\to\bR^7$ unsolved. We briefly mention several other open problems:
\begin{enumerate}
\item Removable singularities: If $\Sigma^n\to \bR^{n+1}\setminus\{0\}$ is two-sided stable minimal and complete away from $\{0\}$, is it flat? (Yes for $n=2$ by work of Gulliver and Lawson \cite{GulliverLawson}.)
\item One-sidedness: Are there one-sided complete stable minimal immersions? (No by work of Ros \cite{Ros:onesided} for $n=2$.)
\item Classify low Morse index\footnote{$\Sigma^n\to\bR^{n+1}$ has Morse index $=I$ if it's stable on the $L^2$-orthogonal complement of an $I$-dimensional space of deformations. Theorem \ref{thm:stab-bern-alldim} can be generalized to prove that complete finite index immersions have particularly simple asymptotic behavior (in particular they must be properly immersed).} minimal hypersurfaces in $\bR^{n+1}$. (See \cite{LopezRos:index-one,ChMa14,ChMa2018} for results in $\bR^3$.) Is the catenoid the unique index $1$ example in $\bR^4$? (See \cite{CMMR:splitting}.)
\item If $\Sigma^n\to\bR^{n+1}$ is a complete two-sided stable minimal immersion where $n \gg 8$, does it have (intrinsic) volume growth estimates $\vol(B^\Sigma_\rho(p)) = O(\rho^n)$? Can it have exponential intrinsic volume growth?
\end{enumerate}

Finally, we briefly mention that the rigidity of (non-compact) stable minimal hypersurfaces has also been extensively studied in other ambient manifolds. For example, Fischer-Colbrie and Schoen proved \cite{Fischer-Colbrie-Schoen} that a complete non-compact two-sided stable minimal immersion $\Sigma \to (M^3,g)$ into a $3$-manifold with non-negative scalar curvature must be conformal to the plane or cylinder and in the cylindrical case, the immersion is totally geodesic, intrinsically flat, and the scalar curvature and normal Ricci curvature both vanish along $\Sigma$ (cf.\ \cite[Proposition C.1]{CCE}). With Eichmair and Moraru, we proved a global rigidity result for $(M^3,g)$ with nonnegative scalar curvature that contains an area-minimizing cylinder \cite{CEV:mincyl}  (see also \cite{Liu:milnor,CCE}), generalizing earlier work of Cai and Galloway who proved similar results for area-minimizing tori \cite{Cai-Galloway:2000}.

In higher dimensions, related results were proven by Shen and Ye \cite{ShenYe,SY:general} later extended and refined by Catino, Mastrolia, and Roncoroni \cite{CMR}. We also obtained similar results with Li and Stryker, proving:
\begin{theorem}[{\cite[Theorem 1.10]{CLS:complete-stable}}]
Suppose that $(M^4,g)$ has weakly bounded geometry, strictly positive scalar curvature $R\geq 1$, and nonnegative sectional curvature $\sec \geq 0$. If $\Sigma^3\to(M^4,g)$ is a complete two-sided stable minimal immersion then it's totally geodesic and has vanishing normal Ricci curvature. 
\end{theorem}
The proof of this result also uses the $\mu$-bubble exhaustion technique to bound the volume growth of certain ends of $\Sigma$ (see also \cite{ChodoshLiStryker:VolumeGrowth}). These ideas have been widely applied to similar problems by several people. Here are some open problems in this vein:
\begin{enumerate}
\item If $\Sigma^4\to (M^5,g)$ is a complete two-sided stable minimal immersion into a manifold with bounded geometry, strictly positive scalar curvature $R\geq 1$, and non-negative sectional curvature $\sec \geq 0$ must $\Sigma$ be totally geodesic with vanishing normal Ricci curvature? This is not true one dimension higher, i.e.\ for $\Sigma^5\to (M^6,g)$, see \cite[Section B.3]{CLS:complete-stable}. One may expect that $\Sigma^4\to (M^5,g)$ is ``borderline'' since it's possible that the stability conditions could be used to imply that $\Sigma^4$ is macroscopically two-dimensional (cf.\ Section \ref{sec:PSC-geo}) and thus might have quadratic area growth. 
\item If $\Sigma^n \to \bH^{n+1}$ is a two-sided stable CMC immersion in hyperbolic space with mean curvature $|H| \geq n$ (the mean curvature of a horosphere) must $\Sigma$ be a horosphere? This was solved by Da Silveira for $n=2$ but is unsolved for $n\geq 3$. Partial progress has been achieved \cite{Cheng:cmc,Deng:CMC,Hang:CMC-H4}. Similar questions may be asked for CMC immersions $\Sigma^n\to \bR^{n+1}$ or $\Sigma^n\to \bS^{n+1}$ (cf.\ \cite{HongYan,ChenHongLi}). 
\end{enumerate}

\section{Generic regularity for the Plateau problem.}\label{sec:generic}  

We've shown that area-minimizing hypersurfaces are effective tools for comparison geometry. The possibility of singularities in high dimensions is a major issue with the technique. We close this survey by discussing generic regularity, which provides one possible approach to this problem based on the following (the name is taken from the recent work \cite{CLZ:biRicci} but the conjecture has been around since at least the 1980s in various forms):

\begin{conjecture}[Generic regularity hypothesis]\label{conj:GRH}
For a closed smooth manifold $M^{n+1}$, there's a Baire generic set of metrics $\mathcal{G}$ so that if $g\in \mathcal{G}$ and $\sigma \in H_n(M;\ZZ)$ is a codimension-one homology class then there's a smooth representative $\Sigma \in \sigma$ of least area.
\end{conjecture}

One may reasonably extend this conjecture include $\mu$-bubbles since they behave very similarly (the integral constraint scales away near a singularity). 

We now explain how the generic regularity hypothesis can be used in minimal surface comparison results by describing the proof of the following result generalizing Theorem \ref{thm:geroch-high-dim}.  
\begin{theorem}\label{theo:torusPMT}
Assume the generic regularity hypothesis is known for all closed manifolds $X^k$ with $k \leq n+1$. If $M^{n+1}$ is a closed manifold then $T^{n+1}\# M^{n+1}$ does not admit a Riemannian metric of positive scalar curvature. 
\end{theorem}
Theorem \ref{theo:torusPMT} is known to imply the positive mass theorem.
\begin{proof}[Sketch of the proof]
We mimic the proof of Theorem \ref{thm:geroch-high-dim}. Assume that $(T\#M, g)$ has positive scalar curvature. Since positivity of scalar curvature is an open condition and the set of metrics $\mathcal{G}$ in the generic regularity hypothesis is dense (by assumption), we may assume without loss of generality that $g \in \cG$. Thus, the area minimizer $\Sigma_n$ in the homology class $T^n\times \{0\} \subset T^{n+1}\# M$ will be smooth. As in Theorem \ref{thm:geroch-high-dim}, $\Sigma_n$ has strictly positive scalar curvature in the spectral sense. Schoen and Yau showed that this implies \cite{SY:manuscripta} that the induced metric on $\Sigma_n$ is conformal to a metric of strictly positive scalar curvature $g_n$. 

Since we assumed the generic regularity hypothesis is applicable to $\Sigma_n$, we can replace $g_n$ by a nearby metric and thus assume that $g_n\in \cG$ (again, this does not change the condition of positivity of scalar curvature). Now, using the topological argument from Theorem \ref{thm:geroch-high-dim}, we may repeat this process (minimize, conformally change the induced metric, perturb to an element of $\cG$) until we reach a contradiction in $n=2$. 
\end{proof}

Progress on Conjecture \ref{conj:GRH} was made by Hardt and Simon \cite{HardtSimon} and Smale \cite{Smale:generic} who proved it for $8$-dimensional ambient manifolds (the first time singularities can appear). Their approach has two essential features:
\begin{enumerate}
\item Perturb the metric $g$ slightly to start. Then, for a (singular) area-minimizer $\Sigma^7 \subset (M^8,g)$, we may perturb $g$ slightly to $\tilde g$ so that the new minimizer $\tilde\Sigma$ moves to ``one side'' relative to $\Sigma$, but remains close to $\Sigma$. 
\item The singularities of $\Sigma$ are modeled on area-minimizing hypercones $\cC$. Hardt and Simon proved that if $\cC$ is singular only at the origin (this always holds in $8$-dimensions) then $\tilde\Sigma$ will be modeled on a \emph{completely smooth} area-minimizing hypersurface $S$, lying on one side of $\cC$. By $\varepsilon$-regularity for area-minimizers, this shows that $\tilde\Sigma$ is completely smooth. 
\end{enumerate}
The first step works in all dimensions but (seemingly) not the second step. If $\cC$ has a non-trivial singular set, then the classification of area-minimizing hypersurfaces lying to one side is not solved (existence was proven by Z.\ Wang \cite{Wang:smoothing}). Even worse, a proof of such a uniqueness result would not substitute for step (2) above, as it might be possible that $\tilde\Sigma$ is modeled on an \emph{iterated tangent cone} of $\cC$ (i.e.\ a singularity of $\tilde\Sigma$ is very close to a singularity of $\cC$ that does not lie at the origin; this is not prohibited by $\varepsilon$-regularity). 

With Mantoulidis and Schulze, we have been able to resolve this issue in $9$ and $10$ ambient dimensions (see also \cite{ChodoshMantoulidisSchulze:Improved910,Li:mink}) and recently in addition with Z.\ Wang we have extended this to $11$-dimensions as well.
\begin{theorem}[{\cite[Theorem 1.2]{CMS:generic.9.10}, \cite[Theorem 1.2]{CMSW}}]\label{thm:generic-reg}
The generic regularity hypothesis holds in manifolds of dimension $\leq 11$. 
\end{theorem}
The proof of Theorem \ref{thm:generic-reg} is too complicated to fully describe here (one may refer to the lecture notes \cite{Chodosh:UCL}). Instead we show how a key ingredient (cf.\ \cite[Proposition 3.3]{CMS:generic.9.10}) is intimately related to De Giorgi's splitting theorem (see step 3 in the outline of the proof of Theorem \ref{theo:bernstein}). 
\begin{proposition}\label{prop:cant-move-cone}
Suppose $\cC^n\subset \bR^{n+1}$ is an area-minimizing hypercone and $\mathbf{x} \in \bR^{n+1}\setminus\{\mathbf{0}\}$ is a non-zero vector so that $\cC \cap (\cC + \mathbf{x}) = \emptyset$. Then $\cC$ is a flat hyperplane. 
\end{proposition}
\begin{proof}[Idea of the proof]
Since $\cC$ is dilation invariant (this is the definition of cone) we find that $\cC \cap (\cC + \lambda \mathbf{x}) = \emptyset$ for all $\lambda > 0$. This proves that the regular part of $\cC$ is graphical in the $\mathbf{x}$-direction and thus we can find a solution to the minimal surface equation \eqref{eq:MSE}. Removable singularity results for the minimal surface equation (and estimates for the singular set of an area-minimizer) imply that this solution extends to all of $\RR^n = \mathbf{x}^\perp$. A non-flat $\cC$ cannot be graphical near the tip of the cone at $\mathbf{0}$. This finishes the proof. 
\end{proof}
\begin{remark}
We cannot help but indicate how this implies De Giorgi's splitting theorem. Suppose that $M$ is a minimal graph in $\bR^{n+1}$ and $\cC$ is a tangent cone at infinity obtained by blowing down $M$. The graphical condition implies that $M + \lambda \mathbf{e}_{n+1}$ is disjoint from $M$ for all $\lambda > 0$, we find that $\cC$ is (weakly) disjoint from $\cC+\mathbf{e}_{n+1}$. The strong maximum principle then implies that either $\cC$ and $\cC+\mathbf{e}_{n+1}$ are strictly disjoint (impossible by Proposition \ref{prop:cant-move-cone}) or else they coincide. This implies that $\cC$ is translation invariant and thus splits a $\bR$ factor. 
\end{remark}
Proposition \ref{prop:cant-move-cone} may be combined with a geometric argument to show that if $\Sigma$ is a minimizer locally modeled on $\cC$ and $\tilde \Sigma$ is a new minimizer lying ``to one side'' then $\tilde \Sigma$ either improves near $\cC$ (is less singular in a certain sense) or else it is closely modeled on some singularity of $\cC$. This does not directly prove that $\tilde\Sigma$ is less singular than $\Sigma$, but what it does imply is that if we form a ``foliation'' $\Sigma_s$ that moves in a monotone way then the singular set $\mathcal{S}$ of the entire foliation has the (Hausdorff) dimension bound $\leq n-7$ (this is what one would obtain for a single leaf!). When combined with an estimate on the rate that (disjoint) minimal hypersurfaces ``pull apart'' at a singularity (going back to Simons \cite{Simons:minimalVarieties} classification of non-flat stable cones) this proves that some leaf of the foliation must be non-singular, proving Theorem \ref{thm:generic-reg} (in $9$ and $10$ dimensions). 

To extend this proof to $11$ dimensions, one must analyze one specific case, namely that of quadratic cylindrical cones, where the argument is borderline. Using Jacobi field analysis and nonconcentration estimates (pioneered by Simon \cite{Simon:cylindrical-sing,Simon:uniqueness.some,Simon:liouville}, cf.\ \cite{Szkelyhidi:cylindrical,Szkelyhidi:unique,EdelenSzekelyhidi:cylindrical}) we prove that either the dimension estimate improves or else the separation estimate improves near such a cone. Results of this form had previously been proven by Figalli, Ros-Oton and Serra for the obstacle problem \cite{FigalliSerra,FROS:ihes}. 

Extending Theorem \ref{thm:generic-reg} to higher dimensions currently presents a major challenge. It seems possible that one must choose the ``foliation'' $\{\Sigma_s\}_s$ more carefully, although it is yet unclear how such a choice can be transmitted to the small scale analysis of singularities. Some other problems related to Theorem \ref{thm:generic-reg} are as follows:
\begin{enumerate}
\item Classify or otherwise improve understanding of minimizing hypercones $\cC^7\subset \bR^8$. Must they be ``strictly minimizing'' and/or ``strictly stable'' in the sense of \cite{HardtSimon}? 
\item Simons proved the non-existence of non-flat stable cones in low enough dimensions by proving that quadratic cones have the smallest first eigenvalue of the stability operator among minimal hypersurfaces in the sphere \cite{Simons:minimalVarieties}. Is it possible to estimate the gap to the non-quadratic cones (assuming they exist)?
\end{enumerate}
Finally, we note that the techniques from \cite{CMS:generic.9.10,CMSW} are likely to be applied in other geometric settings. Indeed, many of the ideas were developed in our work with K.\ Choi on generic mean curvature flow \cite{CCMS:gen.1,CCS:generic2,CCMS:gen.low.ent.1,CMS:genericlowent2,CCMS:revisited}. 

\subsection*{Acknowledgments.}

Thanks to my mentors and collaborators, in particular Simon Brendle, Kyeongsu Choi, Michael Eichmair, Yi Lai, Chao Li, Christos Mantoulidis, Davi Maximo, Paul Minter, Anubhav Mukherjee, Felix Schulze, Douglas Stryker, Zhihan Wang, and Kai Xu. Additional thanks are due Chao Li and Christos Mantoulidis for their feedback on an earlier draft. Financial support was provided by an NSF grant (DMS-2304432) and a Terman Fellowship during the preparation of this article.

\bibliographystyle{amsplain}
\bibliography{bib}

\end{document}